\documentclass{amsart}

\usepackage{graphicx}
\usepackage[centertags]{amsmath}
\usepackage{amsfonts}
\usepackage{amssymb}
\usepackage{amsthm}
\usepackage{newlfont}
\usepackage{hyperref}
\usepackage{mathrsfs} %math text caligraphy
\usepackage{yfonts} %math text goth

\newtheorem*{thm*}{Theorem}
\newtheorem{thm}{Theorem}[section]
\newtheorem{cor}[thm]{Corollary}
\newtheorem{lem}[thm]{Lemma}
\newtheorem{prop}[thm]{Proposition}

\theoremstyle{definition}
\newtheorem{defn}[thm]{Definition}

\theoremstyle{remark}

\numberwithin{equation}{section}

\DeclareSymbolFont{bbold}{U}{bbold}{m}{n}
\DeclareSymbolFontAlphabet{\mathbbold}{bbold}

\newcommand{\N}{\mathbb{N}}
\newcommand{\Z}{\mathbb{Z}}

\newcommand{\C}{\mathbb{C}}

\newcommand{\acts}{\curvearrowright}

\newcommand{\pmp}{p{$.$}m{$.$}p{$.$}}

\newcommand{\pv}{\bar{p}}

\newcommand{\rv}{\bar{r}}

\newcommand{\cP}{\mathcal{P}}

\newcommand{\supp}{\mathrm{supp}}

\newcommand{\dom}{\mathrm{dom}}
\newcommand{\rng}{\mathrm{rng}}

\newcommand{\sH}{\mathrm{H}}
\newcommand{\res}{\restriction}

\newcommand{\Real}{\mathrm{Re}}

\begin{document}

\title[Bernoulli shifts with bases of equal entropy are isomorphic]{Bernoulli shifts with bases of equal entropy are isomorphic}
\author{Brandon Seward}
\address{Courant Institute of Mathematical Sciences, New York University, 251 Mercer Street, New York, NY 10003, U.S.A.}
\email{b.m.seward@gmail.com}
\keywords{Bernoulli shift, isomorphism, finitary isomorphism, Ornstein, entropy, non-amenable}
\subjclass[2010]{37A35}
%\subjclass{37A35} %Entropy and other invariants, isomorphism, classification

\thanks{The author was partially supported by ERC grant 306494 and Simons Foundation grant 328027 (P.I. Tim Austin).}

\begin{abstract}
We prove that if $G$ is a countably infinite group and $(L, \lambda)$ and $(K, \kappa)$ are probability spaces having equal Shannon entropy, then the Bernoulli shifts $G \acts (L^G, \lambda^G)$ and $G \acts (K^G, \kappa^G)$ are isomorphic. This extends Ornstein's famous isomorphism theorem to all countably infinite groups. Our proof builds on a slightly weaker theorem by Lewis Bowen in 2011 that required both $\lambda$ and $\kappa$ have at least $3$ points in their support. We furthermore produce finitary isomorphisms in the case where both $L$ and $K$ are finite.
\end{abstract}
\maketitle

\section{Introduction}

Let $G$ be a countably infinite group and let $(L, \lambda)$ be a standard probability space. The \emph{Bernoulli shift} over $G$ with base $(L, \lambda)$ is the probability space $(L^G, \lambda^G)$ together with the left-shift action of $G$: for $g \in G$ and $x \in L^G$, $g \cdot x$ is defined by the rule $(g \cdot x)(t) = x(g^{-1} t)$ for all $t \in G$. Bernoulli shifts have been an important object of study since the very inception of ergodic theory by Birkhoff and von Neumann in the 1930's. In fact, in those early days von Neumann posed the question as to whether the Bernoulli shifts $\Z \acts (2^\Z, u_2^\Z)$ and $\Z \acts (3^\Z, u_3^\Z)$ are isomorphic (here $n$ denotes $\{0, 1, \ldots, n-1\}$ and $u_n$ denotes the uniform probability measure on $\{0, 1, \ldots, n-1\}$). This question turned out to be surprisingly difficult.

Today, a major open problem is to classify all Bernoulli shifts up to isomorphism for every countably infinite group. Progress on this problem has historically possessed two halves -- on the one hand, using a notion of entropy to obtain the non-isomorphism of Bernoulli shifts of distinct entropies, and, on the other, constructing isomorphisms between Bernoulli shifts of equal entropy. Entropy was first developed by Kolmogorov \cite{Ko58, Ko59} (and corrected and improved by Sinai \cite{Si59}) in 1958 for probability-measure-preserving ({\pmp}) actions of $\Z$. The Kolmogorov--Sinai entropy was extended to {\pmp} actions of countable amenable groups by Kieffer in 1975 \cite{Ki75}. In a major breakthrough in 2008, Bowen \cite{B10b}, together with improvements by Kerr and Li \cite{KL11a}, developed the notion of sofic entropy for {\pmp} actions of sofic groups. Every countable amenable group is sofic, and sofic entropy and Kolmogorov--Sinai entropy coincide for actions of countable amenable groups \cite{B12, KL13}.

For both Kolmogorov--Sinai entropy and sofic entropy, the entropy of Bernoulli shifts $G \acts (L^G, \lambda^G)$ for sofic $G$ has been computed to be equal to the Shannon entropy $\sH(L, \lambda)$ of the base space $(L, \lambda)$ \cite{B10b, KL11b}, where
$$\sH(L, \lambda) = \sum_{\ell \in L} - \lambda(\ell) \log \lambda(\ell)$$
if $\lambda$ has countable support and $\sH(L, \lambda) = \infty$ otherwise. As entropy is an isomorphism invariant, it follows that $G \acts (L^G, \lambda^G)$ and $G \acts (K^G, \kappa^G)$ are non-isomorphic whenever $G$ is sofic (so the sofic entropy is defined) and $\sH(L, \lambda) \neq \sH(K, \kappa)$. In particular, when $G$ is sofic $G \acts (2^G, u_2^G)$ and $G \acts (3^G, u_3^G)$ are non-isomorphic (indeed, it was the work of Kolmogorov that finally settled von Neumann's question).

Its not yet known if all countable groups are sofic. However, if non-sofic countable groups $G$ exist, then it is still unknown whether $(2^G, u_2^G)$ and $(3^G, u_3^G)$ are isomorphic. One route to solving this problem may be to use Rokhlin entropy, which was introduced by the author in 2014 and is defined for {\pmp} actions of general countable groups \cite{S1}. Like sofic entropy, Rokhlin entropy also coincides with Kolmogorov--Sinai entropy for free actions of countable amenable groups \cite{AS}. For free actions of sofic groups, Rokhlin entropy is an upper-bound to sofic entropy and it is open whether the two coincide. The Rokhlin entropy of the Bernoulli shift $G \acts (L^G, \lambda^G)$ is $\sH(L, \lambda)$ when $G$ is sofic but when $G$ is not sofic its value is not yet known. The problem of computing the Rokhlin entropy of Bernoulli shifts over non-sofic groups has strong connections to Gottschalk's surjunctivity conjecture and Kaplansky's direct finiteness conjecture \cite{S2}.

Returning to the second half of the classification problem for Bernoulli shifts, the development of entropy raised the question whether Bernoulli shifts $G \acts (L^G, \lambda^G)$ and $G \acts (K^G, \kappa^G)$ are isomorphic whenever $\sH(L, \lambda) = \sH(K, \kappa)$. Initial positive results for special cases obtained by Me\v{s}alkin were highly motivating \cite{Me59}, but the definitive result appeared in 1970 when Ornstein famously proved this is true for $\Z$ in what has come to be known as ``Ornstein's isomorphism theorem'' \cite{Or70a, Or70b}. In 1975 Stepin defined a group $G$ to be \emph{Ornstein} if $G \acts (L^G, \lambda^G)$ and $G \acts (K^G, \kappa^G)$ are isomorphic whenever $\sH(L, \lambda) = \sH(K, \kappa)$, and Stepin observed that any group containing an Ornstein subgroup is Ornstein \cite{St75}. Then in 1987, Ornstein and Weiss proved that all countably infinite amenable groups are Ornstein \cite{OW87}. Combined with Stepin's observation, it follows that any countable group containing $\Z$ as a subgroup or containing an infinite amenable subgroup must be Ornstein. This covers the vast majority of groups that one tends to encounter.

In 2011, Lewis Bowen made a significant advancement by proving that if $G$ is any countably infinite group and $(L, \lambda)$ and $(K, \kappa)$ are standard probability spaces with $\sH(L, \lambda) = \sH(K, \kappa)$, and one further assumes that the support of $\lambda$ and the support of $\kappa$ each have at least $3$ points, then the Bernoulli shifts $G \acts (L^G, \lambda^G)$ and $G \acts (K^G, \kappa^G)$ are isomorphic \cite{B12b}. In this paper, our main theorem removes Bowen's extra assumption and completes one-half of the classification problem for Bernoulli shifts. To use the terminology of Stepin, we prove that all countably infinite groups are Ornstein.

\begin{thm}
Let $G$ be a countably infinite group and let $(L, \lambda)$ and $(K, \kappa)$ be standard probability spaces. If $\sH(L, \lambda) = \sH(K, \kappa)$ then the Bernoulli shifts $G \acts (L^G, \lambda^G)$ and $G \acts (K^G, \kappa^G)$ are isomorphic.
\end{thm}

When combined with computations of sofic entropy \cite{B10b, KL11b}, this completes the classification of Bernoulli shifts up to isomorphism over sofic groups.

\begin{cor}
Let $G$ be a countably infinite sofic group, and let $(L, \lambda)$ and $(K, \kappa)$ be standard probability spaces. Then the Bernoulli shifts $G \acts (L^G, \lambda^G)$ and $G \acts (K^G, \kappa^G)$ are isomorphic if and only if $\sH(L, \lambda) = \sH(K, \kappa)$.
\end{cor}

Recall that a measure-preserving map $\phi : (X, \mu) \rightarrow (Y, \nu)$ between two topological measure spaces $(X, \mu)$ and $(Y, \nu)$ is \emph{finitary} if there is a conull set $X_0 \subseteq X$ such that the restriction of $\phi$ to $X_0$ is continuous with respect to the subspace topology. We say that $(X, \mu)$ and $(Y, \nu)$ are \emph{finitarily isomorphic} if there is a measure-space isomorphism $\phi : (X, \mu) \rightarrow (Y, \nu)$ such that both $\phi$ and $\phi^{-1}$ are finitary. In case we are considering {\pmp} actions $G \acts (X, \mu)$ and $G \acts (Y, \nu)$, we further require that the map $\phi$ be $G$-equivariant.

Finitary isomorphisms between Bernoulli shifts over $\Z$ having finite base spaces were constructed by Keane and Smorodinsky in 1979 \cite{KS79}. By modifying the proof of our main theorem, we generalize Keane and Smorodinsky's theorem to all countably infinite groups.

\begin{thm}
Let $G$ be a countably infinite group and let $(L, \lambda)$ and $(K, \kappa)$ be standard probability spaces with $L$ and $K$ finite. If $\sH(L, \lambda) = \sH(K, \kappa)$ then the Bernoulli shifts $G \acts (L^G, \lambda^G)$ and $G \acts (K^G, \kappa^G)$ are finitarily isomorphic.
\end{thm}

\subsection*{Outline}
In Section \ref{sec:reduce} we show that it is sufficient to establish isomorphisms for pairs of Bernoulli shifts coming from certain specialized pairs of equal-entropy base spaces. In Section \ref{sec:iso} we prove the main theorem, and in Section \ref{sec:finitary} we construct isomorphisms that are finitary. Finally, in Section \ref{sec:discuss} we discuss the historical origins of the methods of our proof and in particular we articulate how our proof is different from and similar to the prior proof of Bowen that excluded two-atom base spaces.

\section{Reduction to a special case} \label{sec:reduce}

In this section we show that it is sufficient to produce isomorphisms between $(L^G, \lambda^G)$ and $(K^G, \kappa^G)$ for specific pairs of standard probability spaces $(L, \lambda)$ and $(K, \kappa)$ with $\sH(L, \lambda) = \sH(K, \kappa)$. Specifically, we will consider pairs that are related via the relation defined below.

\begin{defn} \label{defn:related}
Fix a non-trivial finite group $\Gamma$. We define a symmetric relation $R_\Gamma$ on the set of standard probability spaces as follows. For two standard probability spaces $(L, \lambda)$ and $(K, \kappa)$ we declare $(L, \lambda) \ R_\Gamma \ (K, \kappa)$ if
\begin{enumerate}
\item [\rm (i)] $\sH(L, \lambda) = \sH(K, \kappa)$;
\item [\rm (i)] the sets $L$ and $K$ are not disjoint;
\item [\rm (ii)] there is a Borel $\Gamma$-invariant set $P \subseteq (L \cap K)^\Gamma$ satisfying the following:
\begin{enumerate}
\item[\rm (a)] $\lambda^\Gamma \res P = \kappa^\Gamma \res P$ and $\lambda^\Gamma(P) > 0$;
\item[\rm (b)] every element of $P$ has trivial $\Gamma$-stabilizer.
\end{enumerate}
\end{enumerate}
\end{defn}

We will show in the next section that if $\Gamma \leq G$, then any pair of $R_\Gamma$-related probability spaces will produce isomorphic Bernoulli shifts over $G$. As isomorphism is transitive, we will consequently obtain isomorphisms between any Bernoulli shifts whose base spaces are equivalent in $[R_\Gamma]^{\text{trans}}$, the transitive closure of $R_\Gamma$. The purpose of this section is to prove the following proposition.

\begin{prop} \label{prop:trans}
Fix a finite group $\Gamma$ with $|\Gamma| \geq 5$. Then for any two standard probability space $(L, \lambda)$ and $(K, \kappa)$ we have $\sH(L, \lambda) = \sH(K, \kappa)$ if and only if $(L, \lambda) \ [R_\Gamma]^{\text{trans}} \ (K, \kappa)$.
\end{prop}

We prove this in a few steps.

\begin{lem} \label{lem:pairequal}
Let $(L, \lambda)$ and $(K, \kappa)$ be probability spaces with $\sH(L, \lambda) = \sH(K, \kappa)$. Suppose there are disjoint sets $A, B \subseteq L \cap K$ such that $\lambda \res A \cup B = \kappa \res A \cup B$ and $\lambda(A), \lambda(B) > 0$. Then $(L, \lambda) \ R_\Gamma \ (K, \kappa)$.
\end{lem}

\begin{proof}
Set $P = \Gamma \cdot \{x \in (L \cap K)^\Gamma : x(1_G) \in A \text{ and } \forall 1_\Gamma \neq \gamma \in \Gamma \ x(\gamma) \in B\}$. Since $A$ and $B$ are disjoint, every point in $P$ has trivial $\Gamma$-stabilizer. Also, $\lambda^\Gamma \res P = \kappa^\Gamma \res P$ since $\lambda$ and $\kappa$ agree on $A \cup B$. Thus all the conditions of Definition \ref{defn:related} are satisfied.
\end{proof}

Below we write $\supp(\lambda)$ for the support of $\lambda$. Also, for a countable Borel partition $\cP$ of $(L, \lambda)$ we write $\sH_\lambda(\cP)$ for the Shannon entropy of $\cP$ with respect to $\lambda$, i.e.
$$\sH_\lambda(\cP) = \sum_{P \in \cP} - \lambda(P) \log(\lambda(P)).$$
Also, if $\pv = (p_i)$ is a probability vector (i.e. a finite or countable tuple of non-negative real numbers that sum to $1$) then we write $\sH(\pv) = \sum_i - p_i \log(p_i)$.

\begin{lem} \label{lem:4more}
Let $(L, \lambda)$ be a probability space with $|\supp(\lambda)| \geq 4$. Then there is $\epsilon > 0$ so that for every $0 < \alpha, \beta < \epsilon$ there is a probability space $(K, \kappa)$ that is $R_\Gamma$-related to $(L, \lambda)$ and there are $a, b \in K$ with $\kappa(a) = \alpha$ and $\kappa(b) = \beta$.
\end{lem}

\begin{proof}
Since $|\supp(\lambda)| \geq 4$ there exists a Borel partition $\{X, Y, M_1, M_2\}$ of $L$ into four sets of positive measure. Set $M = M_1 \cup M_2$. Let $\lambda_X$, $\lambda_Y$, and $\lambda_M$ denote the normalized restrictions of $\lambda$ to $X$, $Y$, and $M$, respectively.  Note that by construction $\sH(M, \lambda_M) > 0$. Fix $0 < \epsilon \leq \lambda(M)/2$ sufficiently small so that
\begin{equation} \label{eq:4more}
\sH \left( \frac{\alpha}{\lambda(M)}, \frac{\beta}{\lambda(M)}, 1 - \frac{\alpha + \beta}{\lambda(M)} \right) \leq \sH(M, \lambda_M)
\end{equation}
whenever $0 \leq \alpha, \beta < \epsilon$.

Now fix $0 < \alpha, \beta < \epsilon$. Then $\alpha + \beta < \lambda(M)$. Inequality (\ref{eq:4more}) implies that we can construct a probability vector $\pv = (p_0, p_1, \ldots)$ with $p_0 = \alpha / \lambda(M)$, $p_1 = \beta / \lambda(M)$, and satisfying $\sH(\pv) = \sH(M, \lambda_M)$. Set $K = X \cup Y \cup \N$ and define the probability measure $\kappa$ by
$$\kappa = \lambda \res (X \cup Y) + \lambda(M) \cdot \sum_{i \in \N} p_i \cdot \delta_i,$$
where $\delta_i$ is the single-point mass on $i$. Setting $a = 0$ and $b = 1$, we clearly have $\kappa(a) = \alpha$ and $\kappa(b) = \beta$. Also, by standard properties of Shannon entropy we have
\begin{align*}
\sH(K, \kappa) & = \sH_\kappa(\{X, Y, \N\}) + \lambda(X) \sH(X, \lambda_X) + \lambda(Y) \sH(Y, \lambda_Y) + \lambda(M) \sH(\pv)\\
 & = \sH_\lambda(\{X, Y, M\}) + \lambda(X) \sH(X, \lambda_X) + \lambda(Y) \sH(Y, \lambda_Y) + \lambda(M) \sH(M, \lambda_M)\\
 & = \sH(L, \lambda).
\end{align*}
Since $\lambda \res X \cup Y = \kappa \res X \cup Y$ and $\lambda(X), \lambda(Y) > 0$, the previous lemma implies that $(K, \kappa)$ is $R_\Gamma$-related to $(L, \lambda)$.
\end{proof}

The following is a technical lemma for handling the case of probability measures supported on only two or three points. We remind the reader that this is the key case to handle due to the prior work of Bowen \cite{B12b}.

\begin{lem} \label{lem:3less}
Let $\pv = (p_0, p_1, \ldots)$ be a probability vector with $p_0 \leq 1/2$ and $p_0 p_1 > 0$ and let $k \geq 4$ be any integer. Then there is a probability vector $\rv = (r_0, r_1, \ldots)$ having at least $4$ strictly positive terms and satisfying $\sH(\rv) = \sH(\pv)$ and $r_0^k \cdot r_1 = p_0^k \cdot p_1$.
\end{lem}

\begin{proof}
Set $r_0 = 1 - 2 p_0^k$ and set
$$r_1 = \left( \frac{p_0}{r_0} \right)^k \cdot p_1 = \left( \frac{p_0}{1 - 2 p_0^k} \right)^k \cdot p_1 \leq \left( \frac{p_0}{1 - 2 p_0^k} \right)^k \cdot (1 - p_0).$$
Then $r_0^k \cdot r_1 = p_0^k \cdot p_1$ as required.

We claim that $r_0 + r_1 < 1$. As $r_0 + r_1 \leq 1 - 2 p_0^k + \frac{p_0^k}{(1 - 2 p_0^k)^k}$, it suffices to show
$$1 - 2 p_0^k + \frac{p_0^k}{(1 - 2 p_0^k)^k} < 1,$$
or equivalently, $(1 - 2 p_0^k)^k > 1/2$. Since $p_0 \leq 1/2$ its enough to check $(1 - 2^{-k+1})^k > 1/2$, or equivalently
$$2^{-1/k} + 2^{-k+1} < 1.$$
The derivative with respect to $k$ of the expression on the left is $\ln(2) k^{-2} 2^{-1/k} - \ln(2) 2^{-k+1}$, which is $0$ when $k^2 2^{1/k} = 2^{k-1}$ or $k \approx 6.59$. So $2^{-1/k} + 2^{-k+1}$ is decreasing for $0 < k < 6.59$, and direct computation verifies that its value is below $1$ when $k = 4$, and for $k > 6.59$ the function $2^{-1/k} + 2^{-k+1}$ is increasing to $1$ as $k \rightarrow \infty$. Thus $2^{-1/k} + 2^{-k+1} < 1$ when $k \geq 4$, as required. We conclude that $r_0 + r_1 < 1$.

Next we claim that the probability vector $(r_0, r_1, 1 - r_0 - r_1)$ satisfies $\sH(r_0, r_1, 1 - r_0 - r_1) < \sH(p_0, 1 - p_0)$. We have
\begin{align*}
\sH(r_0, r_1, 1 - r_0 - r_1) & = \sH(r_0, 1 - r_0) + (1 - r_0) \cdot \sH \left( \frac{r_1}{1 - r_0}, \ \frac{1 - r_0 - r_1}{1 - r_0} \right)\\
 & \leq \sH(1 - 2 p_0^k, 2 p_0^k) + 2 p_0^k \cdot \log(2).
\end{align*}
So it suffices to show that $\sH(1 - 2 p_0^k, 2 p_0^k) + 2 p_0^k \cdot \log(2) < \sH(p_0, 1 - p_0)$.

Define $f : \{z \in \C : 0 < \Real(z) < 1\} \rightarrow \C$ by $f(z) = - z \log(z) - (1 - z) \log(1 - z)$, where $\log$ denotes the standard branch of the logarithm. Direct computation gives that the $n^{\text{th}}$ derivative at $z = 1/2$ is
$$f^{(n)}(1/2) = \begin{cases}
\log(2) & \text{if } n = 0\\
0 & \text{if } n \text{ is odd}\\
-2^n \cdot (n - 2)! & \text{if } n > 0 \text{ is even}.
\end{cases}$$
The function $f$ is a holomorphic function and thus for every $z$ satisfying $|z - 1/2| < 1/2$ we have
\begin{align*}
f(z) & = \sum_{n = 0}^\infty \frac{f^{(n)}(1/2)}{n!} \cdot (z - 1/2)^n\\
 & = \log(2) - 2 (z - 1/2)^2 + \sum_{m = 2}^\infty \frac{-2^{2m}}{2m \cdot (2m - 1)} \cdot (z - 1/2)^{2m}.
\end{align*}
Plugging in $z = p_0$ and $z = 2 p_0^k$ and using the inequalities $p_0 \leq 1/2$ and $k \geq 4$ we obtain
\begin{align*}
 & \sH(p_0, 1 - p_0)\\
 & = \log(2) - 2 (p_0 - 1/2)^2 + \sum_{m = 2}^\infty \frac{-2^{2m}}{2m \cdot (2m - 1)} \cdot (p_0 - 1/2)^{2m}\\
 & = \log(2) - \frac{1}{2} + 2 p_0 - 2 p_0^2 + \sum_{m = 2}^\infty \frac{-2^{2m}}{2m \cdot (2m - 1)} \cdot (p_0 - 1/2)^{2m}\\
 & \geq \log(2) - \frac{1}{2} + p_0 + \sum_{m = 2}^\infty \frac{-2^{2m}}{2m \cdot (2m - 1)} \cdot (p_0 - 1/2)^{2m}\\
 & > \log(2) - \frac{1}{2} + 8 p_0^k + \sum_{m = 2}^\infty \frac{-2^{2m}}{2m \cdot (2m - 1)} \cdot (2 p_0^k - 1/2)^{2m}\\
 & > \log(2) - \frac{1}{2} + 4 p_0^k - 8 p_0^{2k} + 2 p_0^k \cdot \log(2) + \sum_{m = 2}^\infty \frac{-2^{2m}}{2m \cdot (2m - 1)} \cdot (2 p_0^k - 1/2)^{2m}\\
 & = \log(2) - 2 (2 p_0^k - 1/2)^2 + 2 p_0^k \cdot \log(2) + \sum_{m = 2}^\infty \frac{-2^{2m}}{2m \cdot (2m - 1)} \cdot (2 p_0^k - 1/2)^{2m}\\
 & = \sH(1 - 2 p_0^k, 2 p_0^k) + 2 p_0^k \cdot \log(2),
\end{align*}
justifying our claim.

Finally, since $\sH(r_0, r_1, 1 - r_0 - r_1) < \sH(\pv)$, one can divide the remaining mass $1 - r_0 - r_1 > 0$ and extend $(r_0, r_1)$ to a probability vector $\rv = (r_i)$ having at least $4$ strictly positive terms and satisfying $\sH(\rv) = \sH(\pv)$.
\end{proof}

\begin{proof}[Proof of Prop. \ref{prop:trans}]
Fix a finite group $\Gamma$ with $|\Gamma| \geq 5$. Let $(L, \lambda)$ and $(K, \kappa)$ be probability spaces with $\sH(L, \lambda) = \sH(K, \kappa)$. If $|\supp(\lambda)|, |\supp(\kappa)| \geq 4$, then pick $0 < \alpha, \beta < \min(\epsilon_L, \epsilon_K)$, where $\epsilon_L$, $\epsilon_K$ come from applying Lemma \ref{lem:4more} to $(L, \lambda)$ and $(K, \kappa)$, respectively. Let $(L', \lambda')$ be the resulting space $R_\Gamma$-related to $(L, \lambda)$, and let $(K', \kappa')$ be the resulting space $R_\Gamma$-related to $(K, \kappa)$. Then $a, b \in L' \cap K'$, $\lambda'(a) = \kappa'(a) = \alpha > 0$, and $\lambda'(b) = \kappa'(b) = \beta > 0$, so $(L', \lambda')$ and $(K', \kappa')$ are $R_\Gamma$-related by Lemma \ref{lem:pairequal}. We conclude that $(L, \lambda) \ [R_\Gamma]^{\text{trans}} \ (K, \kappa)$ in this case. So it suffices to show that every non-trivial probability space $(L, \lambda)$ is $R_\Gamma$-related to a probability space $(M, \mu)$ with $|\supp(\mu)| \geq 4$.

So lets assume $|\supp(\lambda)| < 4$, as otherwise we are done. Then $\lambda$ is purely atomic and supported on either two or three points. So there are $\ell_0, \ell_1, \ell_2 \in L$ with $\lambda(\{\ell_0, \ell_1, \ell_2\}) = 1$. Set $\pv = (p_0, p_1, p_2)$ where $p_i = \lambda(\ell_i)$. By rearranging indices we may assume that $p_0 \leq 1/2$ and $p_0 p_1 > 0$. Set $k = |\Gamma| - 1 \geq 4$. By Lemma \ref{lem:3less}, there is a probability vector $\rv = (r_0, r_1, \ldots)$ having at least $4$ strictly positive terms and satisfying $r_0^k r_1 = p_0^k p_1$ and $\sH(\rv) = \sH(\pv)$.

Define the probability space $(M, \mu)$ by setting $M = \{\ell_0, \ell_1, m_2, m_3, \ldots\}$ and defining $\mu(\ell_0) = r_0$, $\mu(\ell_1) = r_1$, $\mu(m_i) = r_i$ ($i \geq 2$). Then $|\supp(\mu)| \geq 4$ and $\sH(M, \mu) = \sH(\rv) = \sH(\pv) = \sH(L, \lambda)$. Define
$$P = \{x \in \{\ell_0, \ell_1\}^\Gamma : \ell_1 \text{ occurs precisely once in } x\}.$$
Clearly each $x \in X$ has trivial $\Gamma$ stabilizer and
$$\lambda^\Gamma(x) = \lambda(\ell_0)^{|\Gamma| - 1} \lambda(\ell_1) = p_0^k p_1 = r_0^k r_1 = \mu(\ell_0)^{|\Gamma|-1} \mu(\ell_1) = \mu^\Gamma(x).$$
Thus $(M, \mu) \ R_\Gamma \ (L, \lambda)$.
\end{proof}

\section{The isomorphism theorem} \label{sec:iso}

In this section we prove the main theorem. We first note a group-theoretic fact that allows us, by the previous section, to significantly reduce the difficulty of the main theorem. We remark that, oddly, the lower-bound $5$ in the following lemma happens to be important because we don't know how to prove Lemma \ref{lem:3less}, or a functionally equivalent lemma, without this lower-bound (recall that lemma was applied with $k = |\Gamma| - 1 \geq 4$).

\begin{lem} \label{lem:findsub}
Let $G$ be a countably infinite group. Then either $G$ contains $\Z$ as a subgroup or else there is a finite subgroup $\Gamma \leq G$ with $|\Gamma| \geq 5$.
\end{lem}

\begin{proof}
If any $g \in G$ has order at least $5$ (or infinite) then we are done, as either $\langle g \rangle \cong \Z$ or else $\Gamma = \langle g \rangle$ is finite and has at least $5$ elements. So assume that $|\langle g \rangle| \leq 4$ for all $g \in G$. It follows from work inspired by the Burnside problem\footnote{Note that this does not follow from progress on the Burnside problem itself because we allow different group elements to have different orders, we just require that all orders be bounded by $4$} that every finitely generated subgroup of $G$ must be finite \cite{L, M}. In particular, if $W \subseteq G$ is any set of $5$ elements then $\Gamma = \langle W \rangle$ is a finite group of order at least $5$.
\end{proof}

We are now ready to prove the main theorem.

\begin{thm} \label{thm:iso}
Let $G$ be a countably infinite group. If $(L, \lambda)$ and $(K, \kappa)$ are standard probability spaces with $\sH(L, \lambda) = \sH(K, \kappa)$ then the corresponding Bernoulli shifts $G \acts (L^G, \lambda^G)$ and $G \acts (K^G, \kappa^G)$ are isomorphic.
\end{thm}

\begin{proof}
Stepin observed that the Ornstein isomorphism theorem for Bernoulli shifts over $\Z$ extends immediately to Bernoulli shifts over any group containing $\Z$ as a subgroup \cite{St75}. Thus we are done if $\Z \leq G$. So by Lemma \ref{lem:findsub} we can assume that $G$ possesses a subgroup $\Gamma$ with $5 \leq |\Gamma| < \infty$. Furthermore, since the relation of isomorphism is transitive, by Proposition \ref{prop:trans} it will suffice to consider the case where $(L, \lambda) \ R_\Gamma \ (K, \kappa)$.

Let $P \subseteq (L \cap K)^\Gamma$ be as in Definition \ref{defn:related}. In particular, $\lambda^\Gamma \res P = \kappa^\Gamma \res P$, $\lambda^\Gamma(P) > 0$, $P$ is $\Gamma$-invariant, and every point in $P$ has trivial $\Gamma$-stabilizer. If $\lambda^\Gamma(P) = 1$ then $\lambda$ must have no atoms since every point in $P$ has trivial $\Gamma$-stabilizer and $\Gamma$ is finite. Thus when $\lambda^\Gamma(P) = 1$ we must have that $(L, \lambda)$ and $(K, \kappa)$ are already isomorphic and we are done. So we can assume $0 < \lambda^\Gamma(P) = \kappa^\Gamma(P) < 1$.

Fix a new symbol $*$ not in $L \cup K$, and set $M = (L \cap K) \cup \{*\}$. Define the $G$-equivariant map $\theta_L : L^G \rightarrow M^G$ by
$$\theta_L(x)(g) = \begin{cases}
x(g) & \text{if } (g^{-1} \cdot x) \res \Gamma \in P\\
* & \text{otherwise}
\end{cases}$$
Also define $\theta_K : K^G \rightarrow M^G$ by the exact same rule. One can easily check that $\theta_L$ and $\theta_K$ are $G$-equivariant.

We claim that these two maps push-forward $\lambda^G$ and $\kappa^G$ to the same measure, $(\theta_L)_*(\lambda^G) = (\theta_K)_*(\kappa^G)$. Indeed, let us explicitly describe this common measure. First define a probability measure $\mu_0$ on $M^\Gamma$ by
$$\mu_0 = \lambda^\Gamma \res P + (1 - \lambda^\Gamma(P)) \cdot \delta_{*^\Gamma} = \kappa^\Gamma \res P + (1 - \kappa^\Gamma(P)) \cdot \delta_{*^\Gamma},$$
where $*^\Gamma$ is the element of $M^\Gamma$ that has constant value $*$, and $\delta_{*^\Gamma}$ is the point-mass. Notice that $\mu_0$ is $\Gamma$-invariant. Now for $g \Gamma \in G / \Gamma$ let $\mu_0^{g \Gamma}$ be the probability measure on $M^{g \Gamma} = \{z : g \Gamma \rightarrow M\}$ obtained by translating $\mu_0$, i.e. $\mu_0^{g \Gamma}(g A) = \mu_0(A)$. Since the maps $\theta_L$ and $\theta_K$ are determined by considering the labeling of each $\Gamma$-coset independently, and since $\lambda^G$ and $\kappa^G$ are i.i.d, its straight-forward to see that
$$(\theta_L)_*(\lambda^G) = \prod_{g \Gamma \in G / \Gamma} \mu_0^{g \Gamma} = (\theta_K)_*(\kappa^G),$$
justifying our claim. Let's denote the common measure $\mu = \prod_{g \Gamma \in G / \Gamma} \mu_0^{g \Gamma}$.

We will define an isomorphism $\pi : (L^G, \lambda^G) \rightarrow (K^G, \kappa^G)$ that preserves the common factor $(M^G, \mu)$, i.e. $\pi$ will satisfy $\theta_K \circ \pi = \theta_L$. In other words, any occurrence of a $\Gamma$-coset whose labeling is in $P$ will remain unchanged by $\pi$, while the labels of other $\Gamma$-cosets (the ones labeled with $*$'s in the factor $(M^G, \mu)$) will be overwritten by $\pi$.

We claim that the action $G \acts (M^G, \mu)$ is essentially free. Fix a group element $s \neq 1_G$ and fix $g \in G$. Suppose that $z \in M^G$ satisfies $(g s \cdot z) \res \Gamma = (g \cdot z) \res \Gamma$. If $g s = \gamma g$ with $\gamma \in \Gamma$, then it is necessary that $(g \cdot z) \res \Gamma$ be fixed by $\gamma \neq 1_\Gamma$. Since each labeling in $P$ has trivial stabilizer, it follows that $z \res g^{-1} \Gamma$ must be the constant $*$ function. Alternatively, if $g s \not\in \Gamma g$ then $s^{-1} g^{-1} \Gamma \neq g^{-1} \Gamma$ and the labeling of $z$ on $g^{-1} \Gamma$ determines the labeling of $z$ on $s^{-1} g^{-1} \Gamma$. Therefore the $\mu$-probability that $z \in M^G$ satisfies $(g s \cdot z) \res \Gamma = (g \cdot z) \res \Gamma$ is bounded above by the maximum of $1 - \lambda^\Gamma(P)$ and $\mu_0 \times \mu_0 (\{(r, r) : r \in M^\Gamma\})$, which is strictly less than $1$ and independent of $g$. The condition $(g s \cdot z) \res \Gamma = (g \cdot z) \res \Gamma$ is equivalent to a condition on $z \res (s^{-1} g^{-1} \Gamma \cup g^{-1} \Gamma)$. Since its easy to find an infinite collection $g_i$ with the sets $s^{-1} g_i^{-1} \Gamma \cup g_i^{-1} \Gamma$ pairwise disjoint, and since the labelings on these sets are {i.i.d.} with respect to $\mu$, it follows that $\mu(\{z \in M^G : s \cdot z = z\}) = 0$. This proves our claim.

Consider the set $\{z \in M^G : z(1_G) = *\}$. This set is $\Gamma$-invariant and has positive measure since $\lambda^\Gamma(P) < 1$. As $\Gamma$ is finite, we can pick a Borel set $V \subseteq \{z \in M^G : z(1_G) = *\}$ that meets every $\Gamma$-orbit in $\{z \in M^G : z(1_G) = *\}$ precisely once. Notice that the sets $\gamma \cdot V$, $\gamma \in \Gamma$, are pairwise-disjoint since $G$ acts freely on $(M^G, \mu)$. So $M^G$ is partitioned by the sets $M^G \setminus \Gamma \cdot V$ and $\gamma \cdot V$, $\gamma \in \Gamma$, and we have $M^G \setminus \Gamma \cdot V = \{z \in M^G : z \res \Gamma \in P\}$.

Let $E_G^{M^G} = \{(z, g \cdot z) : z \in M^G, \ g \in G\}$ be the $G$-orbit-equivalence relation on $M^G$. Since $\mu$ is $G$-ergodic and $\mu(V) > 0$, the restricted equivalence relation $E_G^{M^G} \cap V \times V$ on $V$ has infinite classes almost-everywhere. So we can fix an aperiodic element $T$ in the full group $[E_G^{M^G} \cap V \times V]$ \cite[Lem. 3.25]{JKL02} (i.e. $T : V \rightarrow V$ is a Borel bijection whose orbits are infinite almost-everywhere and for almost-every $z \in V$ we have $(z, T(z)) \in E_G^{M^G}$). For each $z \in V$ define $t(z) \in G$ by the rule $t(z) \cdot z = T(z)$ (this is well-defined almost-everywhere).

Set $V_L = \theta_L^{-1}(V)$ and $V_K = \theta_K^{-1}(V)$. Similarly lift $T$ to aperiodic transformations $T_L : V_L \rightarrow V_L$ and $T_K : V_K \rightarrow V_K$ by setting $T_L(x) = t \circ \theta_L(x) \cdot x$ and $T_K(y) = t \circ \theta_K(y) \cdot y$. Notice again we have $L^G$ is partitioned by the sets $L^G \setminus \Gamma \cdot V_L$ and $\gamma \cdot V_L$, $\gamma \in \Gamma$, and that $L^G \setminus \Gamma \cdot V_L = \{x \in L^G : x \res \Gamma \in P\}$ (and similarly for $K^G$ and $V_K$).

Define probability spaces $(A, \alpha)$ and $(B, \beta)$ by setting $A = L^\Gamma \setminus P$, $B = K^\Gamma \setminus P$, and letting $\alpha$ and $\beta$ be the normalized restrictions of $\lambda^\Gamma$ to $A$ and of $\kappa^\Gamma$ to $B$, respectively. Consider the systems $\Z \acts (A^\Z, \alpha^\Z)$ and $\Z \acts (B^\Z, \beta^\Z)$ where the action of $\Z$ is induced by the standard shift map $S : A^\Z \rightarrow A^\Z$, $S(x)(n) = x(n-1)$ (we will abuse notation and also let $S$ denote the shift map on $B^\Z$). Notice that
$$\sH(L^\Gamma, \lambda^\Gamma) = |\Gamma| \cdot \sH(L, \lambda) = |\Gamma| \cdot \sH(K, \kappa) = \sH(K^\Gamma, \kappa^\Gamma)$$
and if we write $(\lambda^\Gamma)_P$ and $(\kappa^\Gamma)_P$ for the normalized restrictions of $\lambda^\Gamma$ and $\kappa^\Gamma$ to $P$, respectively, then
\begin{align*}
\sH(A, \alpha) & = \frac{\sH(L^\Gamma, \lambda^\Gamma) - \sH_{\lambda^\Gamma}(\{P, L^\Gamma \setminus P\}) - \lambda^\Gamma(P) \sH(P, (\lambda^\Gamma)_P)}{1 - \lambda^\Gamma(P)}\\
 & = \frac{\sH(K^\Gamma, \kappa^\Gamma) - \sH_{\kappa^\Gamma}(\{P, K^\Gamma \setminus P\}) - \kappa^\Gamma(P) \sH(P, (\kappa^\Gamma)_P)}{1 - \kappa^\Gamma(P)}\\
 & = \sH(B, \beta).
\end{align*}
Therefore by Ornstein's original isomorphism theorem there is a $\Z$-equivariant isomorphism $\zeta : (A^\Z, \alpha^\Z) \rightarrow (B^\Z, \beta^\Z)$ \cite{Or70a, Or70b}.

We can now intuitively describe the isomorphism $\pi$. For each point $x \in L^G$, we leave the $\Gamma$-cosets that are labeled with a pattern from $P$ unchanged. In the remaining $\Gamma$-cosets, each coset has a unique point in $V_L$, and the transformation $T_L$ arranges these distinguished points into various $\Z$-orbits (in this description we are ignoring the difference between coordinates $g$ of the function $x$ and points $g^{-1} \cdot x$ in the orbit of $x$). For each distinguished point, the labeling of its $\Gamma$-coset is an element of $A$. So along the $\Z$-orbits created by $T_L$, we have bi-infinite sequences of elements from $A$ which we can apply $\zeta$ to in order to obtain bi-infinite sequences of elements of $B$. Each element of $B = K^\Gamma \setminus P$ is a $\Gamma$-labeling, and the newly acquired $B$-labels of the distinguished points tell them what the new labeling of their $\Gamma$-coset should be under $\pi$. We now proceed to define $\pi$ explicitly and verify that it is an isomorphism.

First define $f_L : V_L \rightarrow A^\Z$ by
$$f_L(x)(n)(\gamma) = T_L^{-n}(x)(\gamma) \qquad (n \in \Z, \ \gamma \in \Gamma),$$
and similarly define $f_K : V_K \rightarrow B^\Z$. Notice that $f_L$ is equivariant for $T_L$ and $S$:
$$(f_L \circ T_L(x))(n)(\gamma) = T_L^{1-n}(x)(\gamma) = f_L(x)(n-1)(\gamma) = (S \circ f_L(x))(n)(\gamma).$$
Similarly $f_K$ is equivariant for $T_K$ and $S$. Now define $\pi = \pi_L : (L^G, \lambda^G) \rightarrow (K^G, \kappa^G)$ by
$$\pi_L(x)(g) = \begin{cases}
x(g) & \text{if } g^{-1} \cdot x \not\in \Gamma \cdot V_L\\
\zeta \circ f_L(\gamma^{-1} g^{-1} \cdot x)(0)(\gamma^{-1}) & \text{if } g^{-1} \cdot x \in \gamma \cdot V_L, \ \gamma \in \Gamma.
\end{cases}$$
We similarly define $\pi_K : (K^G, \kappa^G) \rightarrow (L^G, \lambda^G)$, replacing $\zeta$ with $\zeta^{-1}$. It is easily checked that $\pi_L(x)(g) = \pi_L(g^{-1} \cdot x)(1_G)$ and therefore $\pi_L$ (and $\pi_K$) is $G$-equivariant. Note that if $x \not\in \Gamma \cdot V_L$ then $\pi_L(x) \res \Gamma = x \res \Gamma \in P$, while if $x \in \Gamma \cdot V_L$ then $x \res \Gamma \in A = L^\Gamma \setminus P$ and $\pi_L(x) \res \Gamma \in B = K^\Gamma \setminus P$ (and similarly if we swap $L$'s and $K$'s). In particular, $\theta_K \circ \pi_L = \theta_L$  and $\theta_L \circ \pi_K = \theta_K$ as originally desired.

It only remains to check that $\pi = \pi_L$ is an isomorphism. In fact we'll see that $\pi_L^{-1} = \pi_K$. The simplest way to do this is to check that $\pi_L$ induces measure-space isomorphisms between the fibers over $\theta_L$ and the corresponding fibers over $\theta_K$, and that fiber-wise $\pi_K = \pi_L^{-1}$. So let $\lambda^G = \int_{M^G} (\lambda^G)_z \ d \mu(z)$  and $\kappa^G = \int_{M^G} (\kappa^G)_z \ d \mu(z)$ be the disintegrations of $\lambda^G$ and $\kappa^G$ over $\mu$, respectively, and fix $z \in M^G$. We will check that $\pi_L$ induces an isomorphism between $(\lambda^G)_z$ and $(\kappa^G)_z$ with inverse $\pi_K$.

Due to the symmetry of information involved, we point out that anything we assert is true will remain true when $L$ and $K$ (and all of their associated objects and roles) are swapped. To enhance readability, in the remainder of the proof we will often avoid explicitly noting this.

Label the elements of the set $\{g \in G : g \cdot z \in V\}$ as $w_n^i$ ($n \in \Z$, $i \in I$), where $I$ is finite or countable and $T(w_n^i \cdot z) = w_{n+1}^i \cdot z$. Note that $(w_n^i \cdot z) \res \Gamma$ is the constant $*$ function, while if $g \not\in \Gamma \{w_n^i : n \in \Z, \ i \in I\}$ then $(g \cdot z) \res \Gamma \in P$. When $g \not\in \Gamma \{w_n^i : n \in \Z, \ i \in I\}$ we have $(g \cdot x) \res \Gamma = (g \cdot z) \res \Gamma$ for $(\lambda^G)_z$-almost-every $x \in L^G$. Moreover, this agreement is preserved by $\pi_L$. So we only need to pay attention to the joint distributions under $(\lambda^G)_z$ of the labelings $(w_n^i \cdot x) \res \Gamma$, $(n, i) \in \Z \times I$, and similarly for $(\kappa^G)_z$.

By construction, for $(\lambda^G)_z$-almost-every $x \in L^G$, the $w_n^i$ ($n \in \Z$, $i \in I$) are precisely those group elements $g$ for which $g \cdot x \in V_L$, and we have $T_L(w_n^i \cdot x) = w_{n+1}^i \cdot x$. Since under $(\lambda^G)_z$ the labelings $(w_n^i \cdot x) \res \Gamma$, $(n, i) \in \Z \times I$, are {i.i.d.} with distribution $\alpha$, and since
$$f_L(w_0^i \cdot x)(m) = T_L^{-m}(w_0^i \cdot x) \res \Gamma = (w_{-m}^i \cdot x) \res \Gamma,$$
we have that, if $x \in L^G$ is a $(\lambda^G)_z$-random point, then the random variables $f_L(w_0^i \cdot x)$, $i \in I$, and $\zeta \circ f_L(w_0^i \cdot x)$, $i \in I$, are {i.i.d.} with distribution $\alpha^\Z$ and {i.i.d.} with distribution $\beta^\Z$, respectively.

By inspecting the values $\pi_L(x)((w_n^i)^{-1} \gamma)$ in the definition of $\pi_L$, with $\gamma \in \Gamma$, we see that for $(\lambda^G)_z$-almost-every $x \in L^G$
$$(w_n^i \cdot \pi_L(x)) \res \Gamma = \zeta \circ f_L(w_n^i \cdot x)(0).$$
However, $f_L$ is equivariant for $(T_L, S)$ and $\zeta$ is $S$-equivariant, so we have
\begin{align}
(w_n^i \cdot \pi_L(x)) \res \Gamma & = \zeta \circ f_L \circ T_L^n(w_0^i \cdot x)(0)\label{eq:iso}\\
 & = S^n \circ \zeta \circ f_L(w_0^i \cdot x)(0) = \zeta \circ f_L(w_0^i \cdot x)(-n).\nonumber
\end{align}
So it follows from the previous paragraph that under $(\lambda^G)_z$ the labelings $(w_n^i \cdot \pi_L(x)) \res \Gamma \in B$ are {i.i.d.} with distribution $\beta$. This proves that $\pi_L$ pushes $(\lambda^G)_z$ forward to $(\kappa^G)_z$.

Lastly we check that $\pi_K = \pi_L^{-1}$. From (\ref{eq:iso}) we have
$$f_K(w_0^i \cdot \pi_L(x))(m) = T_K^{-m}(w_0^i \cdot \pi_L(x)) \res \Gamma = (w_{-m}^i \cdot \pi_L(x)) \res \Gamma = \zeta \circ f_L(w_0^i \cdot x)(m).$$
Therefore $f_K(w_0^i \cdot \pi_L(x)) = \zeta \circ f_L(w_0^i \cdot x)$. Now, by applying (\ref{eq:iso}) with the roles of $L$ and $K$ swapped, with $\zeta^{-1}$ in place of $\zeta$, and with $\pi_L(x)$ in place of $x$, we obtain 
\begin{align*}
(w_n^i \cdot \pi_K \circ \pi_L(x)) \res \Gamma & = \zeta^{-1} \circ f_K(w_0^i \cdot \pi_L(x))(-n)\\
 & = f_L(w_0^i \cdot x)(-n) = T_L^n(w_0^i \cdot x) \res \Gamma = (w_n^i \cdot x) \res \Gamma
\end{align*}
We conclude that $\pi_K = \pi_L^{-1}$.
\end{proof}

\section{Finitary isomorphisms} \label{sec:finitary}

We will modify the proof of our main theorem by applying the following two lemmas.

\begin{lem} \label{lem:newv}
The set $V$ referred to in the proof of Theorem \ref{thm:iso} can be chosen such that there is a $G$-invariant conull set $Z \subseteq M^G$ with $V \subseteq Z$ clopen in the subspace topology on $Z$.
\end{lem}

\begin{lem} \label{lem:newt}
Let $V$ be as given by the previous lemma. The transformation $T : V \rightarrow V$ referred to in the proof of Theorem \ref{thm:iso} can be chosen such that there is a $G$-invariant conull set $Z' \subseteq M^G$ and a continuous function $t : \Z \times (V \cap Z') \rightarrow G$ such that $t(n, v) \cdot z = T^n(v)$ for all $n \in \Z$ and all $v \in V \cap Z'$.
\end{lem}

Before we prove these lemmas, let's first use them to prove the finitary isomorphism theorem while the proof of Theorem \ref{thm:iso} is still fresh in our minds. Recall that a measure-preserving map between topological measure spaces is finitary if there is a conull set on which the restricted function is continuous.

\begin{thm}
Let $G$ be a countably infinite group and let $(L, \lambda)$ and $(K, \kappa)$ be standard probability spaces with $L$ and $K$ finite. If $\sH(L, \lambda) = \sH(K, \kappa)$ then the Bernoulli shifts $G \acts (L^G, \lambda^G)$ and $G \acts (K^G, \kappa^G)$ are finitarily isomorphic.
\end{thm}

\begin{proof}
The proof mostly follows the proof of Theorem \ref{thm:iso} but with a few modifications. First suppose that $\Z \leq G$. We will detail, and modify, the observation of Stepin \cite{St75} we used before. Say $u \in G$ is of infinite order. Define $q : L^G \rightarrow L^\Z$ by $q(x)(n) = x(u^n)$. By the Keane--Smorodinsky theorem, there is a finitary isomorphism $\psi : (L^\Z, \lambda^\Z) \rightarrow (K^\Z, \kappa^\Z)$. Now define $\pi : (L^G, \lambda^G) \rightarrow (K^G, \kappa^G)$ by
$$\pi(x)(g) = \psi \circ q(g^{-1} \cdot x)(0).$$
It is easily checked that $\pi$ is $G$-equivariant, and since for $w \in G$ and $m \in \Z$ we have
$$\pi(x)(w u^m) = \psi \circ q(u^{-m} w^{-1} \cdot x)(0) = S^{-m} \circ \psi \circ q(w^{-1} \cdot x)(0) = \psi \circ q(w^{-1} \cdot x)(m),$$
its straight-forward to see that $\pi$ pushes $\lambda^G$ forward to $\kappa^G$. Using $\psi^{-1}$ one can make an analogous definition to obtain the inverse to $\pi$. So $\pi$ is an isomorphism. Finally, since $q$ is continuous and $\psi$ is finitary, it follows that $\pi$ is finitary as well.

In the remaining case we assume that $G$ has a subgroup $\Gamma$ satisfying $5 \leq |\Gamma| < \infty$. Looking back at the proof of Proposition \ref{prop:trans}, our argument actually shows that, since $|L|, |K| < \infty$, $(L, \lambda)$ and $(K, \kappa)$ can be connected by a finite sequence of \emph{finite} probability spaces that are sequentially $R_\Gamma$-related. Since finitary isomorphism is transitive, it follows that we can once again assume that $(L, \lambda) \ R_\Gamma \ (K, \kappa)$. The remainder of the proof is the same as before except that we use Lemmas \ref{lem:newv} and \ref{lem:newt} to obtain $V$ and $T$, and instead of obtaining $\zeta : (A^\Z, \alpha^\Z) \rightarrow (B^\Z, \beta^\Z)$ from Ornstein's isomorphism theorem, we instead obtain it from the Keane--Smorodinsky theorem so that it is finitary. We now must only check that the map $\pi$ constructed as before is finitary.

Recall that the maps $\theta_L : (L^G, \lambda^G) \rightarrow (M^G, \mu)$, $f_L : V_L \rightarrow A^\Z$, and $\pi : (L^G, \lambda^G) \rightarrow (K^G, \kappa^G)$ were defined as:
\begin{align*}
\theta_L(x)(g) & = \begin{cases}
x(g) & \text{if } (g^{-1} \cdot x) \res \Gamma \in P\\
* & \text{otherwise}
\end{cases}\\
f_L(x)(n) & = T_L^{-n}(x) \res \Gamma\\
\pi(x)(g) & = \begin{cases}
x(g) & \text{if } g^{-1} \cdot x \not\in \Gamma \cdot V_L\\
\zeta \circ f_L(\gamma^{-1} g^{-1} \cdot x)(0)(\gamma^{-1}) & \text{if } g^{-1} \cdot x \in \gamma \cdot V_L, \ \gamma \in \Gamma.
\end{cases}
\end{align*}
It is evident that $\theta_L$ is continuous. Therefore, since $V$ comes from Lemma \ref{lem:newv}, there is a $G$-invariant conull set $X \subseteq L^G$ so that $V_L = \theta_L^{-1}(V) \subseteq X$ is relatively clopen in $X$. Similarly, as $T$ comes from Lemma \ref{lem:newt}, there is a $G$-invariant conull set $X' \subseteq X$ such that the map $x \in V_L \cap X' \mapsto t(n, \theta_L(x)) \in G$ is continuous for every $n \in \Z$. Since for $x \in V_L \cap X'$ we have
$$f_L(x)(n) = T_L^{-n}(x) \res \Gamma = (t(-n, \theta_L(x)) \cdot x) \res \Gamma,$$
it follows that $f_L$ is continuous on $V_L \cap X'$.

We saw in the proof of Theorem \ref{thm:iso} that the normalized restriction of $\lambda^G$ to $V_L$ gets pushed forward by $f_L$ to $\alpha^\Z$. Since $\zeta$ is finitary there is a conull $\Z$-invariant set $\Omega \subseteq A^\Z$ on which $\zeta$ is continuous. Now let $X''$ be the set of $x \in X' \subseteq L^G$ such that $f_L(g \cdot x) \in \Omega$ whenever $g \cdot x \in V_L$. Then $X''$ is $G$-invariant and conull, and $\zeta \circ f_L$ is continuous on $X'' \cap V_L$. It follows that $\pi$ is continuous on $X''$ and thus $\pi$ is finitary.
\end{proof}

We now turn to the proofs of Lemma \ref{lem:newv} and \ref{lem:newt}. We will rely upon a well-known marker lemma that has appeared in various works. The version we state below comes from \cite[Lem. 3.1]{SS} which itself is a mild generalization of \cite[Lem. 2.1]{GJ}.

\begin{lem} \label{lem:marker}
Let $G \acts X$ be a free continuous action on a zero-dimensional space $X$, let $Y \subseteq X$ be clopen, and let $F \subseteq G$ be finite with $1_G \in F$. Then there is a clopen set $D \subseteq Y$ such that
\begin{enumerate}
\item[\rm (i)] $F d \cap D = \{d\}$ for all $d \in D$; and
\item[\rm (ii)] $Y \subseteq F \cdot D$.
\end{enumerate}
\end{lem}

We can now prove the two lemmas from the start of the section.

\begin{proof}[Proof of Lemma \ref{lem:newv}]
Let $G \acts (M^G, \mu)$ be the factor action constructed in the proof of Theorem \ref{thm:iso}. We argued in that proof that this action is essentially free. So let $Z \subseteq M^G$ be a $G$-invariant conull set on which the action is free everywhere, not just almost-everywhere. Since $M^G$ is zero-dimensional and the action of $G$ on $M^G$ is continuous, we get that $Z$ is zero-dimensional and $G \acts Z$ is continuous. Set $Y = \{z \in Z : z(1_G) = *\}$ and set $F = \Gamma$ and apply Lemma \ref{lem:marker} to get a relatively clopen set $V = D \subseteq Y$. Clause (i) of that lemma implies that $V$ meets every $\Gamma$-orbit in $Y$ at most once, and clause (ii) implies that $V$ meets every $\Gamma$-orbit in $Y$ at least once. Thus $V$ meets every $\Gamma$-orbit in $\{z \in Z : z(1_G) = *\}$ precisely once. So $V$ has the properties required of it in the proof of Theorem \ref{thm:iso} and it is relatively clopen as a subset of the $G$-invariant conull set $Z$.
\end{proof}

\begin{proof}[Proof of Lemma \ref{lem:newt}]
Consider for the moment a free continuous action $G \acts X$. Lets call a function $\phi$ \emph{cocycle-continuous} if $\dom(\phi) \subseteq X$ and there is a continuous function $\bar{\phi} : \dom(\phi) \rightarrow G$ satisfying $\phi(z) = \bar{\phi}(z) \cdot z$ for all $z \in \dom(\phi)$. Notice that, since $G$ acts continuously, cocycle-continuous functions are continuous. Additionally, if $\dom(\phi)$ is clopen then $\phi$ is both an open and a closed map (in particular, the image under $\phi$ of a clopen set is clopen). Also notice that if $\phi$ and $\psi$ are both cocycle-continuous maps and $\rng(\phi) \subseteq \dom(\psi)$ then $\psi \circ \phi$ is cocycle-continuous as well.

As in the previous proof, let $Z \subseteq M^G$ be a $G$-invariant conull set on which the action is free. Let $V \subseteq Z$ be the set constructed in Lemma \ref{lem:newv}. Using the language of the previous paragraph, our goal is to find a $G$-invariant conull set $Z' \subseteq Z$ and build a bijection $T : V \cap Z' \rightarrow V \cap Z'$ having infinite orbits such that every power of $T$ and every power of $T^{-1}$ is cocycle-continuous. By the previous paragraph, it suffices to check that $T$ and $T^{-1}$ are cocycle-continuous.

Set $D_0 = V$ and $F_0 = \{1_G\}$. Fix an increasing sequence $(F_n)_{n \geq 1}$ of finite subsets of $G$ containing the identity and such that, for every $n$, $F_{n+1}$ contains two disjoint translates of $F_n^{-1} F_n$. For each $n \geq 1$, apply Lemma \ref{lem:marker} to obtain a $Z$-relatively clopen set $D_n \subseteq V$ satisfying $F_n^{-1} F_n d \cap D_n = \{d\}$ for all $d \in D_n$ and $V \subseteq F_n^{-1} F_n D_n$. Notice that $F_n d \cap F_n d' = \varnothing$ for $d \neq d' \in D_n$. Also notice that since $F_n^{-1} F_n \cdot v \cap D_n \neq \varnothing$ for all all $v \in V$ and since $F_{n+1}$ contains two disjoint translates of $F_n^{-1} F_n$, it follows that $|D_n \cap F_{n+1} d| \geq 2$ for every $d \in D_{n+1}$ and every $n$.

Fix an enumeration $1_G = g_0, g_1, \ldots$ of $G$. For each $n \in \N$ we define a cocycle-continuous map $\ell_n^{n+1} : D_n \rightarrow D_{n+1}$. Fix $d \in D_n$. First, if there is $f \in F_{n+1}$ with $f^{-1} \cdot d \in D_{n+1}$ (by construction there can be at most one such $f$) then set $\ell_n^{n+1}(d) = f^{-1} \cdot d$. Otherwise, set $\ell_n^{n+1}(d) = g_i \cdot d$ where $i$ is least with $g_i \in F_{n+1}^{-1} F_{n+1}$ and $g_i \cdot d \in D_{n+1}$. Since $D_n$ is contained in the $G$-invariant conull set $Z$, since each $D_{n+1}$ is relatively clopen in $Z$, and since $G$ acts continuously, its immediate that $\ell_n^{n+1}$ is cocycle-continuous. Also, notice that by the last statement of the previous paragraph, $\ell_n^{n+1}$ is everywhere at least two-to-one.

Through the function $\ell_n^{n+1}$ we view each point in $D_n$ as belonging to some point in $D_{n+1}$. Set $\ell_0^n = \ell_{n-1}^n \circ \ell_{n-2}^{n-1} \circ \cdots \circ \ell_0^1$. Define equivalence relations $E_n$ on $V = D_0$ by setting $(v, v') \in E_n$ if and only if $\ell_0^n(v) = \ell_0^n(v')$. Notice that every class of $E_n$ possesses precisely one point from $D_n$. The equivalence relations $E_n$ are Borel, have finite classes, and are increasing, so the union $E = \bigcup_n E_n$ is hyperfinite. Moreover, since $\ell_n^{n+1}$ is everywhere at least two-to-one, each class of $E_n$ has cardinality at least $2^n$, and each class of $E$ is infinite.

Let us intuitively describe how we will construct $T$. On each class of $E_1$ we will choose an ordering and let $T$ move each point to the next one in the ordering. This will leave each class with a ``first'' element where $T^{-1}$ is not defined, and a ``last'' element where $T$ is not defined. Inductively, each class of $E_n$ will have a first and a last point, and repeated application of $T$ will move the first point through all of its $E_n$-class until it reaches the last point. Each $E_n$ class has a unique point in $D_n$, and the function $\ell_n^{n+1}$ indicates how the $E_n$-classes will be combined into $E_{n+1}$-classes, indexed by elements of $D_{n+1}$. For each $d \in D_{n+1}$ we will order the $D_n$ points belonging to $d$, in effect ordering the $E_n$-classes in the $E_{n+1}$-class of $d$. $T$ will then take the last point of an $E_n$-class to the first point of the next $E_n$-class. This will only leave $T$ undefined at the last point of the last $E_n$-class (i.e. the last point of the $E_{n+1}$-class of $d$), and $T^{-1}$ undefined at the first point of the first $E_n$-class (i.e. the first point of the $E_{n+1}$-class of $d$). We then go up one scale and repeat this construction by induction.  

For $n \geq 1$ and $d \in D_n$ set $W_n(d) = \{g \in F_n^{-1} F_n : g \cdot d \in D_{n-1} \text{ and } \ell_{n-1}^n(g \cdot d) = d\}$. Note that $W_n(d)$ is a continuous function of $d$. Using our fixed enumeration of $G$ we obtain an ordering of the set $W_n(d)$. Set $a_n^{n-1}(d) = g \cdot d \in D_{n-1}$ where $g$ is the least element of $W_n(d)$, and set $b_n^{n-1}(d) = g \cdot d \in D_{n-1}$ where $g$ is the largest element of $W_n(d)$. If $g, h \in W_n(d)$ and $h$ is the next element in $W_n(d)$ after $g$, then set $s_n(g \cdot d) = h \cdot d$. So $s_n$ is a function from $D_{n-1} \setminus b_n^{n-1}(D_n)$ to $D_{n-1} \setminus a_n^{n-1}(D_n)$. Note that $a_n^{n-1}$, $b_n^{n-1}$, and $s_n$ are cocycle-continuous functions and that, by remarks in the first paragraph, $a_n^{n-1}(D_n)$ and $b_n^{n-1}(D_n)$ are relatively clopen in $Z$.

Set $a_n^0 = a_1^0 \circ a_2^1 \circ \cdots \circ a_n^{n-1}$ and $b_n^0 = b_1^0 \circ \cdots \circ b_n^{n-1}$ and define $A_n = a_n^0(D_n)$ and $B_n = b_n^0(D_n)$ for $n \geq 1$. Set $B_0 = A_0 = V$. Notice that the $A_n$'s are decreasing, the $B_n$'s are decreasing, and that all of these sets are relatively clopen in $Z$. Now for $v \in B_{n-1} \setminus B_n$ set
$$T(v) = a_{n-1}^0(s_n(\ell_0^{n-1}(v))) \in A_{n-1} \setminus A_n.$$
Equivalently, $T(v) = a_{n-1}^0 \circ s_n \circ (b_{n-1}^0)^{-1}(v)$. In words, if $v$ is the last point of its $E_{n-1}$-class but not the last point of its $E_n$-class, then we lift via $\ell_0^{n-1}$ to the representative $d' \in D_{n-1}$ of its $E_{n-1}$-class, we then apply $s_n$ to shift to the representative $d'' \in D_{n-1}$ of the next $E_{n-1}$-class in its $E_n$-class, and then we apply $a_{n-1}^0$ to travel down from $d''$ to the first point in its $E_{n-1}$-class. The formula for the inverse to $T$ is similarly described: when $v \in A_{n-1} \setminus A_n$ we have $T^{-1}(v) = b_{n-1}^0(s_n^{-1}(\ell_0^{n-1}(v))) = b_{n-1}^0 \circ s_n^{-1} \circ (a_{n-1}^0)^{-1}(v)$.

Since $T$ is a composition of cocycle-continuous functions it is cocycle-continuous on its domain, and we see directly from the formula for $T^{-1}$ that it is cocycle continuous on its domain as well. Lastly, we inspect the domain of $T$. The set $B_n$ meets every $E_n$-class precisely once, and since each $E_n$ class has cardinality at least $2^n$, it follows that $\mu(B_n) \leq 2^{-n}$. The sets $B_n$ are decreasing so their intersection has measure $0$. Similarly the intersection of the $A_n$'s is null. Setting $Z' = Z \setminus G \cdot ((\bigcap_n B_n) \cup (\bigcap_n A_n))$, we have that $Z'$ is $G$-invariant, conull in $M^G$, and that $T$ is a bijection on $V \cap Z'$.
\end{proof}

\section{A comparative discussion of the proof} \label{sec:discuss}

Given that a weaker, but quite similar, version of our main theorem was proven by Bowen a few years ago, we feel its proper to openly identify how our proof both builds upon and differs from Bowen's proof. We'll also take this opportunity to more generally acknowledge the historical roots of some of the methods involved. Specifically, beyond Ornstein's isomorphism theorem, which is a highly celebrated, deep, and very technical achievement, the proof of our main theorem involves two additional main ingredients.

One is to abstractly apply Ornstein's theorem (or a variation of it). This trick was first used by Stepin in 1975 in the case of groups $G$ that contain $\Z$ as a subgroup. One of Bowen's main contributions was to do this in a more abstract way by using a copy of $\Z$ in the full-group of the orbit-equivalence relation rather than a copy of $\Z$ in $G$. We too used this technique (though our copy of $\Z$ was in the pseudo-group, not the full-group; it was generated by the transformation $T : V \rightarrow V$). We remark that in order for this technique to work, it is necessary that $(L^G, \lambda^G)$ and $(K^G, \kappa^G)$ admit a non-trivial common factor and that they each have strong independence properties over this common factor.

The second main ingredient is the reduction to considering only special pairs of equal-entropy probability spaces (like we did with $R_\Gamma$). This clever reduction was first used in Keane and Smorodinsky's proof of the finitary isomorphism theorem for Bernoulli shifts over $\Z$ in 1979, and it was used by Bowen and again by us. We remark that for Bowen and ourselves, the choice of special pairs had to be carefully made in conjunction with the construction of common factors with strong independence properties, as required for the method described in the previous paragraph.

The biggest difference between our proof and Bowen's, and the source of all the smaller differences, is in this last ingredient. In contrast with our relation $R_\Gamma$, Bowen considered pairs $(L, \lambda)$ and $(K, \kappa)$ such that $\sH(L, \lambda) = \sH(K, \kappa)$ and such that there exists a third space $(M, \mu)$ that is non-trivial and admits measure-preserving maps $(L, \lambda) \rightarrow (M, \mu)$ and $(K, \kappa) \rightarrow (M, \mu)$. By applying these maps coordinate-wise, one then sees that $(M^G, \mu^G)$ is a common factor of $(L^G, \lambda^G)$ and $(K^G, \kappa^G)$, just as we similarly constructed a common factor, and just as is required as we mentioned above. However, it is evident that this requirement that the base spaces have a non-trivial factor immediately rules out the possibility of considering two-atom base spaces. Our main contribution here, therefore (ignoring the result on finitary isomorphisms), is in the creation of the relations $R_\Gamma$, the use of lesser-known group theoretic facts (Lemma \ref{lem:findsub}), and in particular the technical proof of Lemma \ref{lem:3less}.

Lastly, differences in the construction of the common factor $M^G$ led Bowen to use a transformation $T : M^G \rightarrow M^G$ on all of $M^G$, while we used a transformation on a subset $T : V \rightarrow V$. Additionally, due to differences in the independence properties present, Bowen had to invoke Thouvenot's relative isomorphism theorem and apply it to the transformations lifted from $T$, while our independence properties were cleaner and we were able to simply apply Ornstein's isomorphism theorem to the lifted transformations. As an unintended bonus, our cleaner independence properties furthermore allowed us to invoke the Keane--Smorodinsky finitary isomorphism result and obtain finitary isomorphisms in general.

\end{document}